\documentclass[11pt,a4paper]{amsart}
\usepackage[T1]{fontenc}
\usepackage[latin9]{inputenc}
\usepackage{amssymb}
\usepackage{color}
\usepackage{graphicx,color}
\usepackage{amsmath, amssymb, graphics}
\usepackage{hyperref}

\makeatletter
\numberwithin{equation}{section} 
\numberwithin{figure}{section} 
  \@ifundefined{theoremstyle}{\usepackage{amsthm}}{}
  \theoremstyle{plain}
  \newtheorem{thm}{Theorem}[section]
  \theoremstyle{plain}
  
  \theoremstyle{plain}
  \newtheorem{prop}[thm]{Proposition}
  \theoremstyle{remark}
  \newtheorem{rem}[thm]{Remark}
  \theoremstyle{plain}
 \newtheorem{mydef}[thm]{Definition}
  \theoremstyle{remark}
  
  \theoremstyle{plain}
  \newtheorem{lem}[thm]{Lemma}

\smallskip
\def\<{{\langle }}
\def\>{{\rangle }}

\smallskip
\def\<{{\langle }}
\def\>{{\rangle }}

\begin{document}

\title[Lagrange point of Euler solutions]{Lagrange points of Euler's solutions of the 3-body problem}
\author{Oscar Perdomo}
\date{\today}

\maketitle
\centerline{\scshape \'Oscar Perdomo\footnote{Corresponding author} }
{\footnotesize
 \centerline{ Department of Mathematics}
   \centerline{Central Connecticut State University}
   \centerline{New Britain, CT 06050, USA}
}


\begin{abstract}

In this paper we classify the central configurations of the circular restricted 4-body problem with three primaries at the collinear configuration of the 3-body problem and an infinitesimal mass. The case where the three primaries have the same mass, with one of the bodies staying motionless at the center of mass, was considered in 2021 by Llibre. The video \url{https://youtu.be/PWFtqxd4RUA} goes over some of the results in this paper.

\end{abstract}
\begin{center}\rule{0.9\textwidth}{0.1mm}
\end{center}





\section{Introduction} In this paper we will consider the $n$-body problem. If we assume that the $i^{th}$ body has mass $m_i$ and its  motion  is described  by the vector function $X_i(t)$, then these functions satisfy the following system of $n$ differential equations of order two:

$$m_i \ddot{X}_i=\sum_{j\ne i} \frac{G m_i m_j}{|X_j-X_i|^3}(X_j-X_i) $$ 

where $G$ is a constant known as the gravitational constant. By canceling $m_1$ in the first equation, $m_2$ in the second equation, and so on, we get the system

\begin{eqnarray}\label{enbp} 
\ddot{X}_i=\sum_{j\ne i} \frac{G m_j}{|X_j-X_i|^3}(X_j-X_i)
\end{eqnarray}

These basic cancellations allow us to consider the case where one of the masses is zero, leading to what is known as the restricted $n$-body problem. Usually, the $n-1$ bodies with non-zero masses are called the {\it primaries}. The restricted $n$ body problem is useful to study the motion of asteroids and probes in the presence of planets, moons, and stars. The 2-body problem has been completely solved. The general solution for the $n$-body problem for  $n>2$ is an open problem and it has been responsible for big developments in the theory of dynamical systems.

In 1770, in search of explicit solutions of the three body problem, Euler classified all the solutions of the three body problem where three bodies are collinear and move on concentric circles with the same constant angular velocity, \cite{E}. In particular, when one of the three masses is zero, Euler's solutions provide the solutions of the restricted 3-body problem known as the Lagrange points $L_1$, $L_2$, and $L_3$. Two years later, Lagrange discovered the restricted solutions of the three body problem known as the Lagrange points $L_4$ and $L_5$, \cite{L}.  

In this paper we will find all the solutions of the restricted 4 body problem where the three primaries are the solutions discovered by Euler, and a fourth body with mass zero moves on a circle with the same center, at the same angular velocity, as that of the primaries. We call these solutions {\it Euler-Lagrange solutions} or   {\it Euler-Lagrange points}.

We will see that the family of solutions found by Euler can be described with two parameters. For one of these solutions, the one where the three bodies have the same mass and  the body in the middle stays fixed, Llibre found that there are 6 Euler-Lagrange points, four of which are collinear with the primaries \cite{L1}. Later on, Llibre, Pasca, and Valls studied the stability of these six Euler-Lagrange points, \cite{L2}.

Euler-Lagrange solutions are particular examples of central configurations. For a review on central configurations we refer to \cite{M}. For results regarding the existence of  collinear central configurations we refer to \cite{P} and \cite{Mou}. There are many results for central configurations of the 4-body problem; refer to the references in the \cite{L1}


This paper is organized as follows: Section \ref{tbp} reviews Euler's solutions. The presentation that we do here is slightly different to the one shown by Euler and it shows how these solutions depend on two parameters, one of them, $m_3$, being the mass of one of the bodies.  Since the Lagrange points will play an important role in the description of the Euler-Lagrange points, in Section \ref{cL} we will explain how to deduce the Lagrange points $L_1$, $L_2$ and $L_3$ from  the Euler's solutions and we also show that, even though there is no formula in terms of radicals to find them, due to the fact that they are given in terms of a quintic equation, there exists a relation between them in terms of radicals. Section \ref{LeP} explain all the Euler-Lagrange points. We divide this last section into two subsections. The first one deals with the case when one of primaries remains fixed at the center of gravity (as Llibre did) but we allow the masses to be different. The second subsection computes the Euler-Lagrange points in the general case, when all of the primaries move.

\section{Euler's solutions of the three-body problem}\label{tbp}

In this section we classify all the collinear solutions of the of the three-body problem with masses $m_1$, $m_2$, and $m_3$ that move in circles around the origin. This work was done by Euler, \cite{E}, and for the sake of completeness and notation we deduce these solutions here in a slightly different way. Without loss of generality, by changing the unit of mass,  we can assume that the gravitational constant is one. We also assume, by changing the unit of time, that the angular velocity is one and finally, by changing the unit of distance, we will assume that the body that is alone on one side is at distance one from the center of mass. More precisely, we assume that for some $r_2$ and $r_3$ with $0<r_2<r_3$,

\begin{eqnarray}
X_1&=&(\cos t,\sin t)\nonumber\\
X_2&=&-r_2 (\cos t,\sin t)\label{systemE}\\
X_3&=&-r_3 (\cos t,\sin t)\nonumber
\end{eqnarray}

From Equation (\ref{enbp}) we obtain 6 equations that reduce to the following three:

\begin{eqnarray}\label{inieqs1}
\frac{m_2}{(r_2+1)^2}+\frac{m_3-(r_3+1)^2}{(r_3+1)^2}&=&0\\
\label{inieqs2} \frac{r_2 (r_2+1)^2-m_1}{(r_2+1)^2}+\frac{m_3}{(r_2-r_3)^2}&=&0\\
\label{inieqs3} \frac{r_3 (r_3+1)^2-m_1}{(r_3+1)^2}-\frac{m_2}{(r_2-r_3)^2}&=&0
\end{eqnarray}

Using equation (\ref{inieqs1}) and(\ref{inieqs2}) we get that 

\begin{eqnarray}\label{m1}
m_1&=&\frac{(r_2+1)^2 \left(m_3+r_2 (r_2-r_3)^2\right)}{(r_2-r_3)^2}\\
\label{m2} m_2&=&(r_2+1)^2 \left(1-\frac{m_3}{(r_3+1)^2}\right)
\end{eqnarray}

From Equation (\ref{m1}) we get that $m_1$ is always positive for any $m_3\ge0$. From Equation (\ref{m2}) we get that $m_3$ must satisfy $0\le m_3\le (r_3+1)^2$ and moreover, if $m_3=0$ then $X_3$ provides a solution of the restricted $3$ body problem. That is,  $-r_3$ reduces to a Lagrange point with primaries located at $-r_2$ and $1$. Likewise, if $m_3= (r_3+1)^2$, then $m_2$ is zero and $-r_2$ becomes a Lagrange point. Replacing equations (\ref{m1}) and (\ref{m2}) into Equation (\ref{inieqs3}) we get that 

\begin{eqnarray}
p(r_2,r_3)&=&-r_2^5+2 r_2^4 r_3-2 r_2^4-r_2^3 r_3^2+4 r_2^3 r_3-r_2^3+r_2^2 r_3^3-r_2^2 r_3^2+r_2^2 r_3 \label{poly}\\
& &-r_2^2 - 2 r_2 r_3^4-4 r_2 r_3^3-5 r_2 r_3^2-4 r_2 r_3-2 r_2+r_3^5\nonumber \\
& & +2 r_3^4+r_3^3-r_3^2-2 r_3-1=0 \nonumber
\end{eqnarray}

Let us show that for values of $r_2\ge0$ and $r_3>0$ the level set $p(r_2,r_3)$ defines a function in term of $r_2$. Figure \ref{levelsetp} shows the graph of this function. 

\begin{lem} \label{df} The relation $p(r_2,r_3)=0$ given in Equation (\ref{poly}) defines a function $r_3=f(r_2)$ for values of $r_2\ge 0$. We have that $f(0)=1$.
\end{lem}
\begin{proof} This result follows from the fact that 

\begin{eqnarray*}
\frac{\partial p}{\partial r_3}&=&2 r_2^4-2 r_2^3 r_3+4 r_2^3+3 r_2^2 r_3^2-2 r_2^2 r_3+r_2^2-8 r_2 r_3^3-12 r_2 r_3^2\\
& &-10 r_2 r_3-4 r_2+5 r_3^4+8 r_3^3+3 r_3^2-2 r_3-2 
\end{eqnarray*}
and the system $p(r_2,r_3)=0=\frac{\partial p}{\partial r_3}$ has no real solution with $r_2\ge0$.

\end{proof}

\begin{figure}[h]
\centerline
{\includegraphics[scale=0.59]{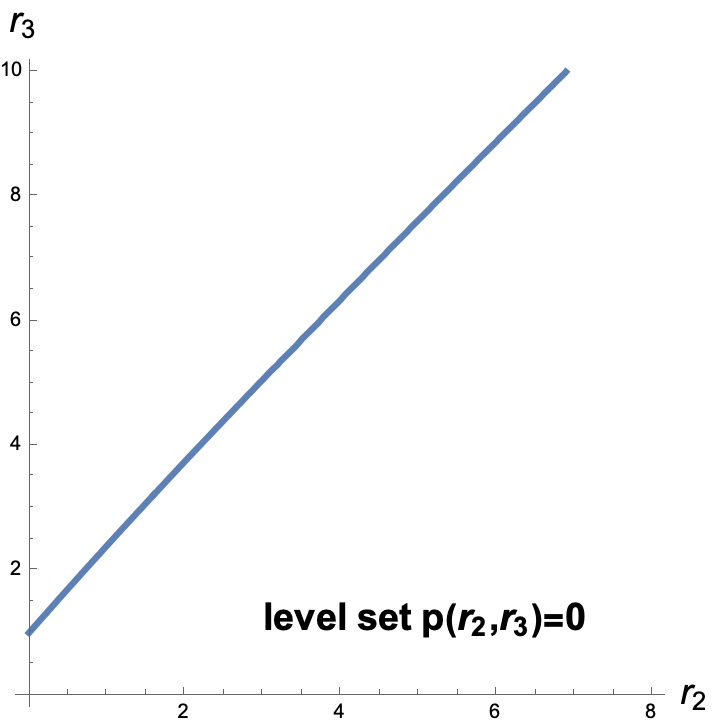}}
\caption{Relation between the radii $r_2$ and $r_3$.}
\label{levelsetp}
\end{figure}

\begin{prop}\label{ES}
For any $r_2\ge0$ and any $m_3$ between $0$ and $(1+f(r_2))^2$, there exists a unique solution of the 3-body problem with masses $m_1$, $m_2$ and $m_3$ and positions given by 
Equations (\ref{systemE}). The masses $m_1$ and $m_2$ are determined by $m_3$. See Equations (\ref{m1}) and (\ref{m2}). We call this solution of the 3-body problem $ES(r_2,m_3)$.
\end{prop}

\begin{rem}\label{remLagrangePoints}
All the computations that we have done allow $m_2$ or $m_3$ to be zero. Therefore, we can use the relation $p(r_2,r_3)=0$ to compute the Lagrange points $L_1$, $L_2$ and $L_3$. We will explain this in Section \ref{cL}.
\end{rem}

\begin{rem}
For circular solutions of the 2-body problem, we have that the ratio of the radii is determined by the ratio of the masses. We lose this property for circular solutions of the 3-body problem. Every triple  $(1,r_2,r_3)$ with $r_3=f(r_2)$ has a one-parametric family of masses that can be placed in the positions given by this triple. For example, approximating to 11 significant digits, we have that the solution associated with

$$(r_1,r_2,r_3)=(1.00000000000, 2.00000000000, 3.74714216664) $$

admits the masses 

$$m_1=20.9483973941, \, m_2= 8.60062761371, \, m_3=1.00000000000$$

and the masses

$$m_1=26.8451921822, \,m_2= 7.80188284113,\,m_3=  3.00000000000$$

See Figure \ref{twosetofmasses}

\begin{figure}[h]
\centerline
{\includegraphics[scale=0.59]{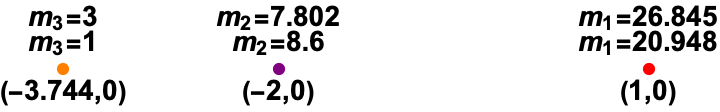}}
\caption{Two sets of masses that can be placed at the same points.}
\label{twosetofmasses}
\end{figure}
\end{rem}

\section{Lagrange points}\label{cL}

As pointed out in the previous section, we can deduce Lagrange points $L_1$, $L_2$ and $L_3$ from the Euler solutions. We have the following Lemma,

\begin{lem}
Let $x>0$ and let us assume that two bodies with masses $m_1$ and $m_2$ and positions  $(\cos(t),\sin(t))$ and $(-x \cos(t),-x\sin(t))$ respectively satisfy the 2-body problem.  If we define $L_1(x)\dots L_5(x)$ by the following relations

\begin{eqnarray*}
(L_1(x),0) & &\text{is the  coordinate of the Lagrange point between the two primaries}\\
(L_2(x),0) & & \text{ is the coordinate of the Lagrange point with $L_2(x)>1$}\\
(L_3(x),0)  & &\text{ is the coordinate of the Lagrange point with $L_3(x)<-x$}\\
L_4(x) \, \text{ and }\,  L_5(x) & & \text{are the  coordinates of the non-collinear Lagrange}
\end{eqnarray*}

and $f(x)$ is the function defined in Lemma \ref{df}, then we have that:

\begin{eqnarray}
L_1(x) & =&\begin{cases}-f^{-1}(x) &x\ge1\\ x f^{-1}(1/x)& x<1 \end{cases}\label{L1}\\
L_2(x) &= & xf(1/x)\label{L2}\\
L_3(x)  &= &-f(x)\label{L3}\\
L_4(x)  &= &(\frac{1-x}{2},\frac{1-x}{2}\, \sqrt{3})\label{L4}\\
L_5(x)  &= &(\frac{1-x}{2},-\frac{1-x}{2}\, \sqrt{3})\label{L5}
\end{eqnarray}
\end{lem}

\begin{proof}
We will use Proposition \ref{ES} in this proof. The expressions for $L_4$ and $L_5$ follow because these points and the primaries form equilateral triangles. To deduce the expression for  $L_1$ we notice that if $x\ge 1$, $ES(r_2,m_3)=ES(f^{-1}(x),(1+x)^2)$ is a solution with $m_2=0$, and $r_3=x$. Therefore $r_2= f^{-1}(x)$ must be a Lagrange point. Since this Lagrange point is between the two primaries,  then $L_1(x)=f^{-1}(x)$ when $x\ge1$. When $0<x<1$ we have that $ES(r_2,m_3)=ES(f^{-1}(\frac{1}{x}),(1+\frac{1}{x})^2)$ is a solution with $m_2=0$, and $r_3=\frac{1}{x}$. By changing the orientation of the first axis and by dilating by $x$ (by changing the units of distance) we get that $(1,0)$, $(x f^{-1}(\frac{1}{x}),0)$ and $(-x,0)$ are initial positions for the collinear circular solution of the 3-body problem. Since the position of the body with mass zero is $(x f^{-1}(\frac{1}{x}),0)$  and this point is between the two primaries, then if  $0<x<1$, $L_1(x)=x f^{-1}(\frac{1}{x})$.  We use a similar argument to deduce that $L_2(x)= xf(\frac{1}{x})$. We have that $ES(r_2,m_3)=ES(\frac{1}{x},0)$ is a solution with $m_3=0$, and $r_3=f(\frac{1}{x})$.
 By changing the orientation of the first axis and by dilating by $x$ (by changing the units of distance) we get that $(xf(1/x),0)$, $(1,0)$ and $(-x,0)$ are initial positions for the collinear circular solution of the 3-body problem. Since the position of the body with mass zero is $(xf(1/x),0)$  and this point is to the right of the primary located at $(1,0)$, then $L_2(x)=x f(\frac{1}{x})$. 
$L_3(x)= -f(x)$, because $ES(r_2,m_3)=ES(x,0)$ is a solution with $m_3=0$, and $r_3=f(x)$, therefore $(-f(x),0)$ is a Lagrange point with the primaries located a $(-x,0)$ and $(1,0)$.

\end{proof}

The previous result states that we can compute the Lagrange points $L_1$, $L_2$ and $L_3$ by understanding the curve 

$$r_2\longrightarrow (r_2,f(r_2))$$

The following theorem tells us that we can parametrize this curve using rational functions.

\begin{thm} There exists a function $G(w)$ in terms of radicals such that

$$w\longrightarrow(G(w),G(w)+w+1)\quad \text{with $w\ge0$}$$

produces all the points of the form $(r_2,r_3)$ with $r_2\ge0$ and $p(r_2,r_3)=0$.

\end{thm}

\begin{proof}
A direct verification shows that if we change $r_3=w+u+1$ and  $r_2=u  $, the Equation $p(r_2,r_3)=0$ reduces to 

\begin{eqnarray*}
0&=&-u^4-2 u^3 w-6 u^3+3 u^2 w^3+8 u^2 w^2+u^2 w-10 u^2+3 u w^4+\\
& &16 u w^3+28 u w^2+14 u w-5 u+w^5+7 w^4+19 w^3+24 w^2+12 w
\end{eqnarray*}

Since this is a polynomial equation of degree 4 in $u$, for one of the solutions we have that
$r_2=u=G(w)$. This solution is the branch that satisfies $G(0)=0$. Since $r_2=u=G(w)$, then $r_3=w+G(w)+1$.
\end{proof}


\section{Euler-Lagrange points}\label{LeP}

We consider solutions of the four body problem of the form

\begin{eqnarray*}
X_1&=&(\cos t,\sin t)\\
X_2&=&-r_2 (\cos t,\sin t)\\
X_3&=&-r_3 (\cos t,\sin t)\\
X_4&=&r_4 (\cos t,\sin t)+r_5 (-\sin t,\cos t)
\end{eqnarray*}

where the body with position $X_i$ has mass $m_i$ and we assume that $m_4=0$. Since the fourth body has mass zero, then adding this additional body does not affect the motion of the other three bodies. A solution of the new system is an addition of an Euler solution $ES(r_2,m_3)$. In particular, we still have that the 
$r_3$ can be written as a function of $r_2$, see Figure \ref{levelsetp}, and 
 $m_1$ and $m_2$ can be written in terms of $m_3$, see Equations (\ref{m1}) and (\ref{m2}), where $m_3$ has the following range:

\begin{eqnarray}\label{m3}
0\le m_3 \le (1+r_3)^2
\end{eqnarray}

Somehow we already have been using the following definition:

\begin{mydef}\label{E-L}
If $X_1(t),\dots,X_4(t)$ is a solution of the $4$-body problem with  the $X_i$'s given above, then we say that $(r_4,r_5)$ is an Euler-Lagrange point of the solution $ES(r_2,m_3)$.
\end{mydef}


A direct computation shows that the equation $\ddot{X}_4=\sum_{i=1}^3 m_i \frac{X_i-X_4}{|X_i-X_4|^3}$ reduces to the following two equations 
\begin{eqnarray}
f_1(r_2,r_3,r_4,r_5,m_1,m_2,m_3)&=&0\label{eq7}\\
f_2(r_2,r_3,r_4,r_5,m_1,m_2,m_3)&=&0\label{eq8}
\end{eqnarray}
with
{\small
\begin{eqnarray*}
f_1&=&\frac{m_1r_4}{\left(r_4^2-2 r_4+r_5^2+1\right)^{3/2}}+\frac{m_2r_4}{\left(r_2^2+2 r_2 r_4+r_4^2+r_5^2\right)^{3/2}}+\frac{m_3r_4}{\left(r_3^2+2 r_3 r_4+r_4^2+r_5^2\right)^{3/2}} \\
& &-r_4-\frac{m_1}{\left(r_4^2-2 r_4+r_5^2+1\right)^{3/2}}+\frac{m_2 r_2}{\left(r_2^2+2 r_2 r_4+r_4^2+r_5^2\right)^{3/2}}+\frac{m_3 r_3}{\left(r_3^2+2 r_3 r_4+r_4^2+r_5^2\right)^{3/2}}
\end{eqnarray*}
}

and
{\small
\begin{eqnarray*}
f_2&=&r_5 \left(-\frac{m_1}{\left(r_4^2-2 r_4+r_5^2+1\right)^{3/2}}-\frac{m_2}{\left(r_2^2+2 r_2 r_4+r_4^2+r_5^2\right)^{3/2}}-\frac{m_3}{\left(r_3^2+2 r_3 r_4+r_4^2+r_5^2\right)^{3/2}}+1\right)
\end{eqnarray*}
}


We divide the study of the system $f_1=0=f_2$ into two parts.

\subsection{The case $r_2=0$} In this section we will consider the case when one of the bodies stays fixed in the center of mass. This case was considered by Llibre in \cite{L1} in the particular case where all the masses are equal. From Equations (\ref{m1}) and (\ref{m2}) we have that 

\begin{itemize}
\item
when $m_3$ approaches zero then $m_1$ approaches zero as well and $m_2$ approaches 1. See Figure \ref{m1m3near0}
\item
 when $m_3$ approaches $(1+r_3)^2=4$, $m_1$ approaches 4 as well, and $m_2$ approaches 0. In this case we expect the Euler-Lagrange points to approach the  Lagrange points when the masses of the two primaries are the same.  See Figure \ref{m1m3near4}
 
 \item
 when $m_3=4/5$, then $m_1=m_2=m_3=4/5$. In this case the Lagrange points of Euler's solution must agree with those found by Llibre in \cite{L1}.

\end{itemize}

\begin{figure}[h]

\includegraphics[scale=0.59]{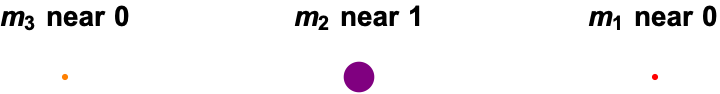}
\caption{This is one of the extreme cases, $m_3\to 0$.}
\label{m1m3near0}
\end{figure}

\begin{figure}[h]
\vskip.2cm\includegraphics[scale=0.59]{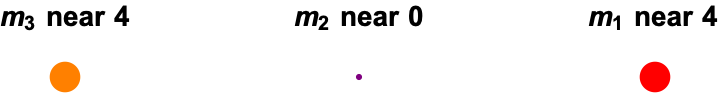}
\caption{This is the other extreme case, $m_3\to 4$.}
\label{m1m3near4}
\end{figure}

\begin{figure}[h]
\vskip.2cm\includegraphics[scale=0.59]{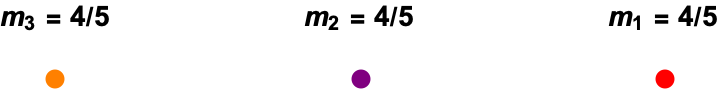}
\caption{All masses are the same. This case was previously studied.}
\label{m1eqm2eqm3}
\end{figure}

We need to solve Equations \ref{eq7} and \ref{eq8}.


\subsubsection{Subcase $r_5=0$}  Equation (\ref{eq8}) holds true because $r_5$ is a factor of the function  $f_2$. Replacing the expressions for $m_1$ and $m_2$ in terms of $m_3$ in the Equation \ref{eq7}, we obtain that 
\begin{itemize}
\item
When $r_4>1$

$$ m_3=\frac{4 (r_4-1)^3 (r_4+1)^2 \left(r_4^2+r_4+1\right)}{7 r_4^4+10 r_4^2-1}=g_1(r_4)$$

We are only interested in values of $r_4$ for which $0<g_1(r_4)<4$. Figure \ref{graphg1} shows this graph. Notice that since $f(1)=2.3968...$, then $(2.3968...,0)$ is one of the Lagrange points when we only have two bodies with the same mass located at $(-1,0)$ and $(1,0)$. This is, $L_2(1)=2.3968...$ using Equation \ref{L2}.

\begin{figure}[h]
\vskip.2cm\includegraphics[scale=0.59]{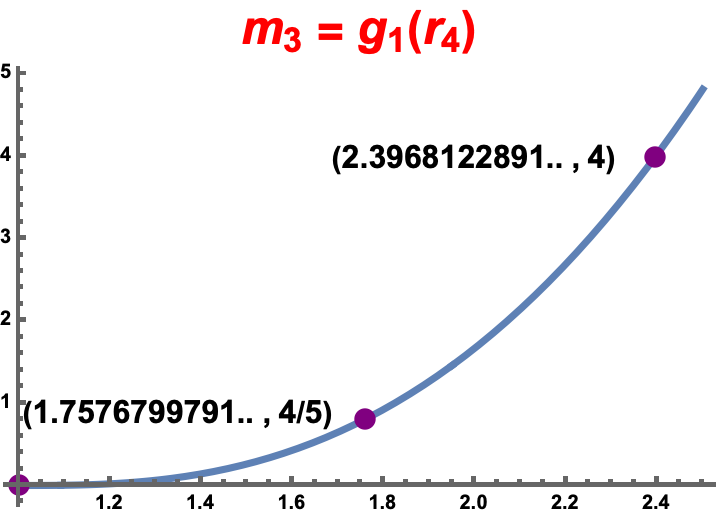}
\caption{Every value of $r_4$ between 1 and $2.3968122...$ is an Euler-Lagrange point if we take $m_3=g_1(r_4)$. The point $(1.7576...,4/5)$ is the one found in \cite{L1}.}
\label{graphg1}
\end{figure}

\item
When $0<r_4<1$

$$m_3= -\frac{4 (r_4-1)^3 (r_4+1)^2 \left(r_4^2+r_4+1\right)}{r_4^4+16 r_4^3-2 r_4^2+1} = g_2(r_4)$$

As expected, $m_3$ takes values between $0$ and $4$. Figure \ref{graphg2} shows the graph of this function.

\begin{figure}[h]
\vskip.2cm\includegraphics[scale=0.59]{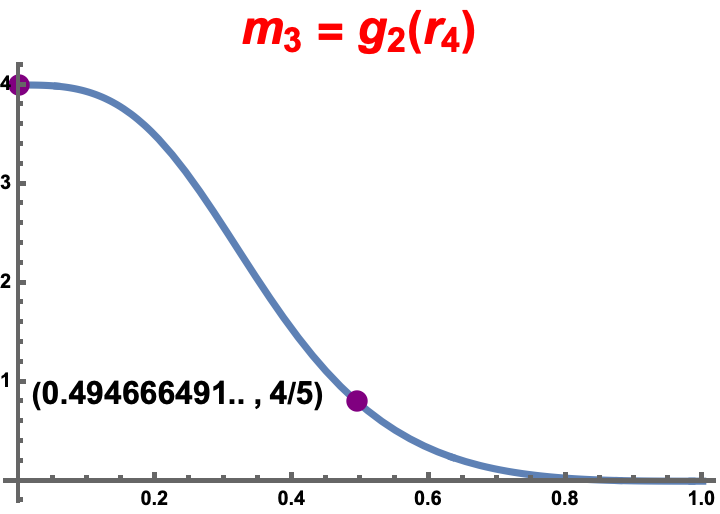}
\caption{Every value of $r_4$ between 0 and $1$ is an Euler-Lagrange point if we take $m_3=g_2(r_4)$. The point $(0.494666491...,4/5)$ is the one found in \cite{L1}.}
\label{graphg2}
\end{figure}

\item
When $-1<r_4<0$

$$ m_3=\frac{4 (r_4-1) (r_4+1) \left(r_4^2-1\right) \left(r_4^3+1\right)}{r_4^4-16 r_4^3-2 r_4^2+1}=g_3(r_4)=g_2(-r_4)$$

The graph of this function is just the reflection about the $m_3$ axis of the graph $m_3=g_2(r_4)$.

\item
When $r_4<-1$

$$m_3= -\frac{4 (r_4-1) (r_4+1) \left(r_4^2-1\right) \left(r_4^3+1\right)}{7 r_4^4+10 r_4^2-1}=g_4(r_4)=g_1(-r_4)$$

The graph of this function is just the reflection about the $m_3$ axis of the graph $m_3=g_1(r_4)$.
\end{itemize} 

\subsubsection{Subcase $r_5\ne0$}  We can check that, after replacing the expression for $m_1$ and $m_2$ in terms of $m_3$, the equation  $f_2(m_1,m_2,m_3,r_2=0,r_3=1,r_4,r_5)=0$ reduces to the equation 

\begin{eqnarray}\label{eq8r5d0}
r_5h(r_4,r_5,m_3)=0
\end{eqnarray}

where

$$h(r_4,r_5,m_3)=\frac{m_3-4}{4 \left(r_4^2+r_5^2\right)^{3/2}}-\frac{m_3}{\left(r_4^2+2 r_4+r_5^2+1\right)^{3/2}}-\frac{m_3}{\left(r_4^2-2 r_4+r_5^2+1\right)^{3/2}}+1$$

and the equation $f_1(m_1,m_2,m_3,r_2=0,r_3=1,r_4,r_5)=0$ reduces to the equation

\begin{eqnarray}\label{eq7r5d0}
\frac{m_3}{\left(r_4^2+2 r_4+r_5^2+1\right)^{3/2}}-\frac{m_3}{\left(r_4^2-2 r_4+r_5^2+1\right)^{3/2}}-r_4 h(r_4,r_5,m_3)=0
\end{eqnarray}

Combining Equations (\ref{eq8r5d0}) and (\ref{eq7r5d0}) we get that 

$$\frac{m_3}{\left(r_4^2+2 r_4+r_5^2+1\right)^{3/2}}-\frac{m_3}{\left(r_4^2-2 r_4+r_5^2+1\right)^{3/2}}=0$$

which implies that $r_4=0$. Then we must have that $h(r_4=0,r_5,m_3)$ must vanish. A direct computation shows that
 $h(r_4=0,r_5,m_3)$  reduces to
 
 $$\frac{m_3-4}{4 \left(r_5^2\right)^{3/2}}-\frac{2 m_3}{\left(r_5^2+1\right)^{3/2}}+1=0$$
 
 We get that any time $(0,r_5)$ is a Lagrange point, then $(0,-r_5)$ is also a Lagrange point. Assuming that $r_5>0$ we get from the equation above that 
 
 $$m_3=\frac{4 \left(r_5^2+1\right)^{3/2} \left(1-r_5^3\right)}{\left(r_5^2+1\right)^{3/2}-8 r_5^3}=g_5(r_5)$$ 

Since $m_3$ can only take values between $0$ and $4$, then $r_5$ can only take values between $1$ and $\sqrt{3}$. Figure \ref{graphg5} shows the graph of the function $g_5$

\begin{figure}[h]
\vskip.2cm\includegraphics[scale=0.59]{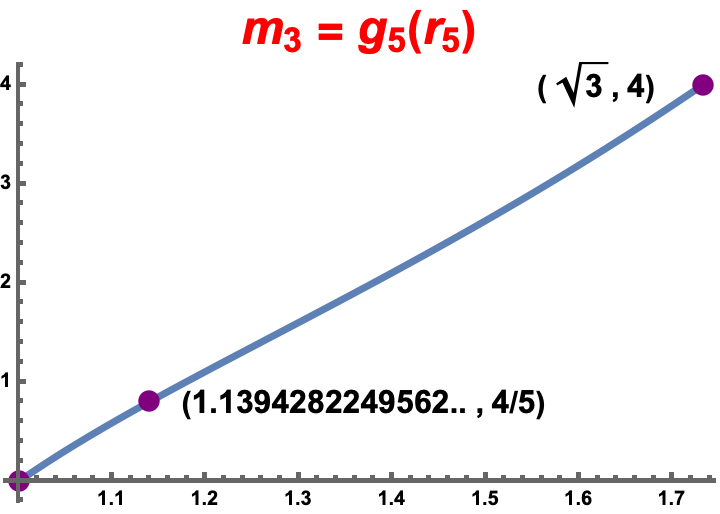}
\caption{Every point of the form $(0,\pm r_5)$ with $r_5$ between 1 and $\sqrt{3}$ is a Euler-Lagrange point if we take $m_3=g_5(r_5)$. The point $(1.1394282249562...,4/5)$ is the one found in \cite{L1}}
\label{graphg5}
\end{figure}

We summarize the computations above in the following theorem.

\begin{thm}
Each Euler's solution with $r_2=0$, in other words, each $ES(0,m_3)$, has six Euler-Lagrange points of the form

$$(\pm x_1,0),\quad (\pm x_2,0), \quad (0,\pm x_3)$$

with $x_1\in (0,1)$, $x_2\in (1,f(1))=(1,L_2(1))$ and $x_3\in (1,\sqrt{3})$. We are using the function $f(x)$ defined in Lemma \ref{df} and the function $L_2$ defined in (\ref{L2}). Moreover, for any $m_3\in (0,4)$ we have that $x_1=g_1^{-1}(m_3)$, $x_2=g_2^{-1}(m_3)$ and $x_3=g_5^{-1}(m_3)$.

\end{thm}

\subsection{The case $r_2>0$ } In this subsection we consider the Euler-Lagrange points for solutions where all the primaries move on circles. None of the bodies stays fixed.
\subsubsection{Subcase $r_2>0$ and $r_5=0$}
Equation (\ref{eq8}) holds true because $r_5$ is a factor of the function  $f_2$. Replacing the expressions for $m_1$ and $m_2$ in terms of $m_3$ in the Equation \ref{eq7}, we obtain that:

\begin{itemize}
\item
When $r_4<-r_3$ we get the $m_3=h_1(r_4)$ with

$$h_1(r_4)=-\frac{\frac{r_2 (r_2+1)^2}{(r_4-1)^2}+\frac{(r_2+1)^2}{(r_2+r_4)^2}+r_4}{\frac{(r_2+1)^2}{(r_4-1)^2 (r_2-r_3)^2}-\frac{(r_2+1)^2}{(r_3+1)^2 (r_2+r_4)^2}+\frac{1}{(r_3+r_4)^2}}$$

Notice that $h_1(-f(r_3))$ must be $(1+r_3)^2$ because when $m_3=(1+r_3)^2$ we have that $m_2=0$, leaving only two primaries, one at $(1,0)$ and the other at $(-r_3,0)$. Since $-f(r_3)=L_3(r_3)$ is the only Lagrange point to the left of $-r_3$ for these two primaries, then $h_1(-f(r_3))=(1+r_3)^2$. More precisely, we have that the only solution $h_1(r_4)=(1+r_3)^2$ with $r_4<-r_3$ is $r_4=-f(r_3)$. Using a similar argument we must have that the only solution of the equation $h_1(r_4)=0$ with $r_4\le -r_3$ must be $r_4=-r_3$. This time $m_3=0$, leave only two primaries, one at $(1,0)$ and the other at $(-r_2,0)$.  Recall that the only Lagrange point with $r_4\le -r_3$ is $r_4=-r_3$. We conclude that for any $m_3\in (0,(1+r_3)^2)$ there is an Euler-Lagrange point  $r_4$ between $-f(r_3)$ and $-r_3$. Figure \ref{graphh1} shows the graph of the function $m_3=h_1(r_4)$ when $r_2=2$.

\item
When $r_4$ is between $-r_3$ and $-r_2$. We have that $m_3=h_2(r_4)$ with

$$h_2(r_4)=\frac{\frac{r_2 (r_2+1)^2}{(r_4-1)^2}+\frac{(r_2+1)^2}{(r_2+r_4)^2}+r_4}{-\frac{(r_2+1)^2}{(r_4-1)^2 (r_2-r_3)^2}+\frac{(r_2+1)^2}{(r_3+1)^2 (r_2+r_4)^2}+\frac{1}{(r_3+r_4)^2}} $$

Similar arguments show that the only solution of the equation $h_2(r_4)=(1+r_3)^2$ with $-r_3\le r_4\le -r_2$ is $r_4=-r_2$ and the only solution of the equation $h_2(r_4)=0$ with $-r_3\le r_4\le -r_2$ is $r_4=-r_3$. We conclude that for any $m_3\in (0,(1+r_3)^2)$ there is an Euler-Lagrange point with $r_4$ between $-r_3$ and $-r_2$. See Figure \ref{graphh2}. 

\item
When $r_4$ is between $-r_2$ and $1$, we have that $m_3=h_3(r_4)$ with

$$h_3(r_4)=\frac{\frac{r_2 (r_2+1)^2}{(r_4-1)^2}-\frac{(r_2+1)^2}{(r_2+r_4)^2}+r_4}{-\frac{(r_2+1)^2}{(r_4-1)^2 (r_2-r_3)^2}-\frac{(r_2+1)^2}{(r_3+1)^2 (r_2+r_4)^2}+\frac{1}{(r_3+r_4)^2}}$$

Similar arguments as in the previous item show that the only solution of the equation $h_3(r_4)=(1+r_3)^2$ with $-r_2\le r_4\le 1$ is $r_4=-r_2$ and the only solution of the equation $h_3(r_4)=0$ with $-r_2\le r_4\le 1$ is $r_4=L_1(r_2)$. We conclude that for any $m_3\in (0,(1+r_3)^2)$ there is an Euler-Lagrange point with $r_4$ between $-r_2$ and $L_1(r_2)$. See Figure \ref{graphh3}. 
\item
When $r_4>1$, we have that $m_3=h_4(r_4)$ with

$$h_4(r_4)=\frac{-\frac{r_2 (r_2+1)^2}{(r_4-1)^2}-\frac{(r_2+1)^2}{(r_2+r_4)^2}+r_4}{\frac{(r_2+1)^2}{(r_4-1)^2 (r_2-r_3)^2}-\frac{(r_2+1)^2}{(r_3+1)^2 (r_2+r_4)^2}+\frac{1}{(r_3+r_4)^2}}$$

Similar arguments as before show that the only solution of the equation $h_3(r_4)=(1+r_3)^2$ with $r_4\ge1$ is $r_4=L_3(r_3)$ and the only solution of the equation $h_3(r_4)=0$ with $r_4>1$ is $r_4=L_1(r_2)$. We conclude that for any $m_3\in (0,(1+r_3)^2)$ there is an Euler-Lagrange point  with $r_4$ between $L_2(r_2)$ and $L_2(r_3)$. See Figure \ref{graphh4}. 

\end{itemize}

 \begin{figure}[h]
\vskip.2cm\includegraphics[scale=0.59]{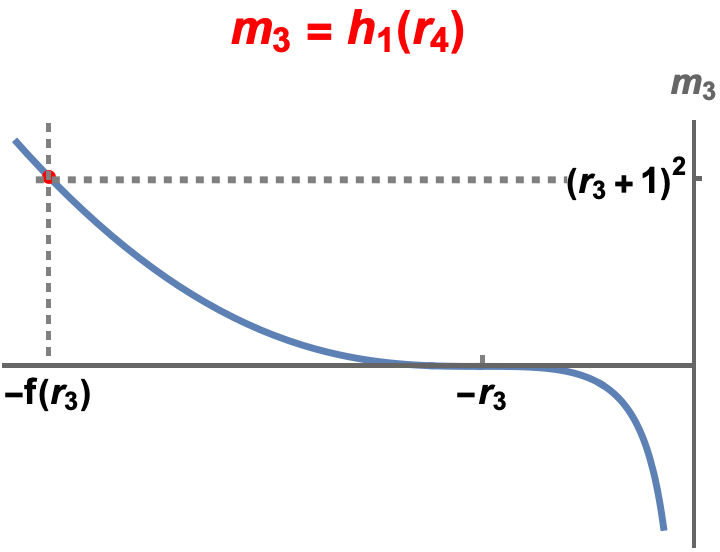}
\caption{For this graph $r_2=2$, $r_3=f(2)=3.7471...$ and $f(r_3)=6.0305...$. Every point of the form $(r_4,0)$ with $r_4$ between $-f(r_3)$ and $-r_3$ is an Euler-Lagrange point if we take $m_3=h_1(r_4)$. }
\label{graphh1}
\end{figure}

 \begin{figure}[h]
\vskip.2cm\includegraphics[scale=0.59]{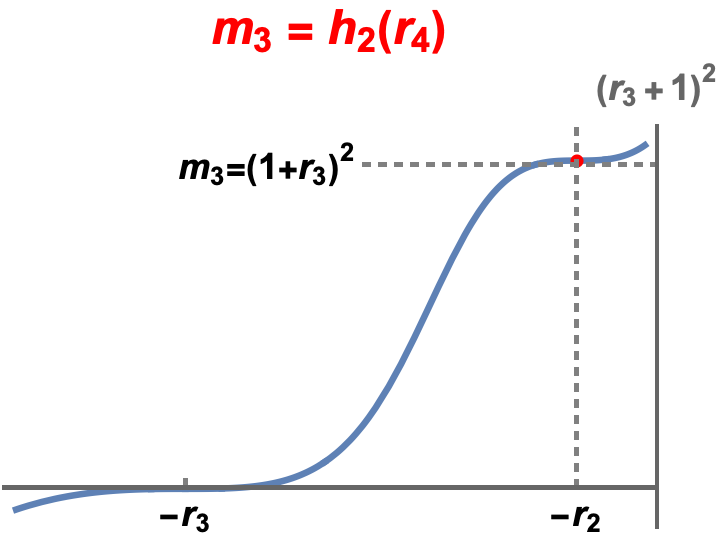}
\caption{For this graph $r_2=2$ and $r_3=f(2)=3.7471...$. Every point of the form $(r_4,0)$ with $r_4$ between $-r_3$ and $-r_2$ is an Euler-Lagrange point if we take $m_3=h_2(r_4)$. }
\label{graphh2}
\end{figure}

 \begin{figure}[h]
\vskip.2cm\includegraphics[scale=0.59]{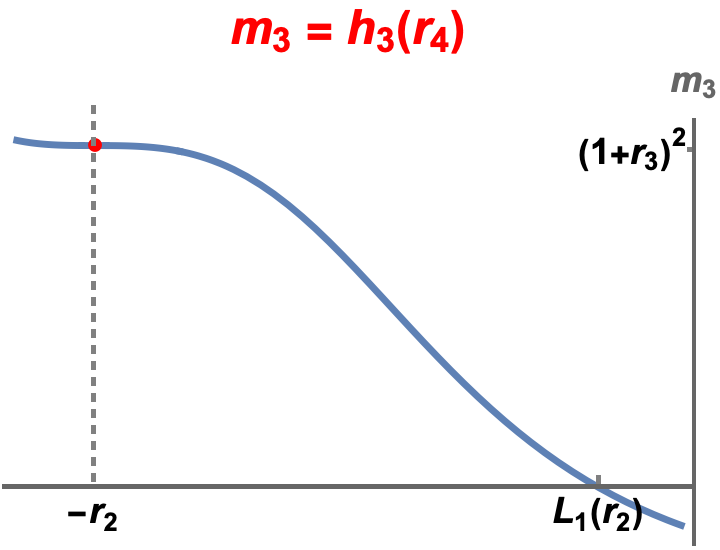}
\caption{For this graph $r_2=2$ and $L_1(r_2)=-0.71225...$  Every point of the form $(r_4,0)$ with $r_4$ between $-r_2$ and $L_1(r_2)$ is an Euler-Lagrange point if we take $m_3=h_3(r_4)$. }
\label{graphh3}
\end{figure}

 \begin{figure}[h]
\vskip.2cm\includegraphics[scale=0.59]{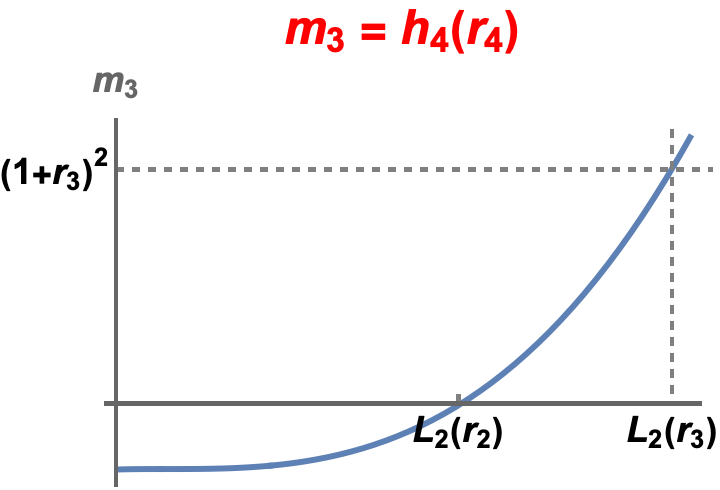}
\caption{For this graph $r_2=2$, $r_3=f(2)=3.7471...$,  $L_2(r_2)=3.4090...$ and $L_2(r_3)=5.161...$  Every point of the form $(r_4,0)$ with $r_4$ between $L_2(r_2)$ and $L_2(r_3)$ is an Euler-Lagrange point if we take $m_3=h_4(r_4)$. }
\label{graphh4}
\end{figure}

\subsubsection{The case $r_2>0$ and $r_5\ne0$} In this subsection we consider the Euler-Lagrange points that are not collinear with the primaries. We can check that Equation (\ref{eq8}) and Equation (\ref{eq7}), this is, the equations  $f_2(m_1,m_2,m_3,r_2r_3,r_4,r_5)=0$  and $f_1(m_1,m_2,m_3,r_2,r_3,r_4,r_5)=0$, respectively,  reduce to the equations

\begin{eqnarray}\label{feqs1}
r_5q_1=0\quad\text{and}\quad q_2+r_4 q_1
\end{eqnarray}

where 

{\scriptsize
\begin{eqnarray}
  q_1&=&-\frac{m_1}{\left(r_4^2-2 r_4+r_5^2+1\right)^{3/2}}-\frac{m_2}{\left(r_2^2+2 r_2 r_4+r_4^2+r_5^2\right)^{3/2}}-\frac{m_3}{\left(r_3^2+2 r_3 r_4+r_4^2+r_5^2\right)^{3/2}}+1\\
  q_2&=&-\frac{m_1}{\left(r_4^2-2 r_4+r_5^2+1\right)^{3/2}}+\frac{m_2 r_2}{\left(r_2^2+2 r_2 r_4+r_4^2+r_5^2\right)^{3/2}}+\frac{m_3 r_3}{\left(r_3^2+2 r_3 r_4+r_4^2+r_5^2\right)^{3/2}}
\end{eqnarray}
}

Since we are assuming that $r_5\ne0$, then we need to solve the system $q_1=q_2=0$. We obtain the equation for the Euler-Lagrange points by replacing $m_1$ and $m_2$ using Equations (\ref{m1}) and (\ref{m2}) in the expressions $q_1=0$ and $q_2=0$ and then solving for $m_3$ in each of the equations. We obtain that $m_3=q_3(r_2,r_3,r_4,r_5)$ and  $m_3=q_4(r_2,r_3,r_4,r_5)$ with

{\scriptsize
\begin{eqnarray}\label{eqq3}
q_3=-\frac{\frac{(r_2+1)^2}{\left(r_2^2+2 r_2 r_4+r_4^2+r_5^2\right)^{3/2}}+\frac{r_2 (r_2+1)^2}{\left(r_4^2-2 r_4+r_5^2+1\right)^{3/2}}-1}{-\frac{(r_2+1)^2}{(r_3+1)^2 \left(r_2^2+2 r_2 r_4+r_4^2+r_5^2\right)^{3/2}}+\frac{(r_2+1)^2}{(r_2-r_3)^2 \left(r_4^2-2 r_4+r_5^2+1\right)^{3/2}}+\frac{1}{\left(r_3^2+2 r_3 r_4+r_4^2+r_5^2\right)^{3/2}}}
\end{eqnarray}
}

{\scriptsize
\begin{eqnarray}\label{eqq4}
q_4=\frac{r_2 (r_2+1)^2 \left(\frac{1}{\left(r_4^2-2 r_4+r_5^2+1\right)^{3/2}}-\frac{1}{\left(r_2^2+2 r_2 r_4+r_4^2+r_5^2\right)^{3/2}}\right)}{-\frac{r_2 (r_2+1)^2}{(r_3+1)^2 \left(r_2^2+2 r_2 r_4+r_4^2+r_5^2\right)^{3/2}}-\frac{(r_2+1)^2}{(r_2-r_3)^2 \left(r_4^2-2 r_4+r_5^2+1\right)^{3/2}}+\frac{r_3}{\left(r_3^2+2 r_3 r_4+r_4^2+r_5^2\right)^{3/2}}}
\end{eqnarray}
}

Therefore, the Euler-Lagrange points with $r_5\ne0$ must satisfy the relation $q_3=q_4$. If $m_3=(1+r_3)^2$ then $m_2=0$, and then we only have two primaries at $(1,0)$ and $(-r_3,0)$.  Therefore the points
$L_4(r_3)$ and $L_5(r_3)$ satisfy the equation $q_3=q_4$. Likewise, if $m_3=0$ then, we only have two primaries at $(1,0)$ and $(-r_2,0)$. Therefore $L_4(r_2)$ and $L_5(r_2)$ must also satisfy the equation $q_3=q_4$. See Figure \ref{relforr5d0}.

 \begin{figure}[h]
\vskip.2cm\includegraphics[scale=0.59]{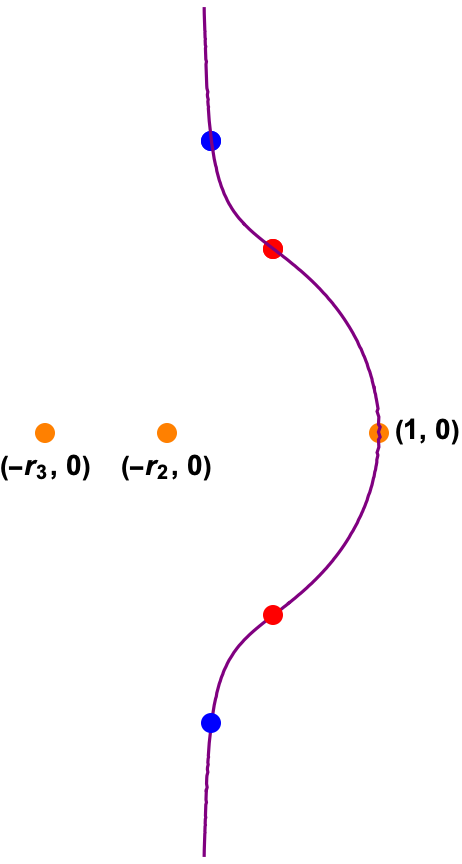}
\caption{This graph shows points in the plane that satisfy the equation $q_3=q_4$. For this example we have taken $r_2=2$ and $r_3=f(2)=3.7471...$. We have plotted the points $L_4(r_3)$ and $L_5(r_3)$ as well as $L_4(r_2)$ and $L_5(r_2)$. }
\label{relforr5d0}
\end{figure}

We summarize the computations above in the following theorem.

\begin{thm}
For any $r_2>0$, and any $m_3\in (0,(1+r_3)^2)$ with $r_3=f(r_2)$, there are $6$ Euler-Lagrange points distributed as follows

\begin{itemize}
\item
One of the form $(r_4,0)$ with $L_3(r_3)<r_4<-r_3$ satisfying  $m_3=h_1(r_4)$
\item
One of the form $(r_4,0)$ with $-r_3<r_4<-r_2$ satisfying  $m_3=h_2(r_4)$
\item
One of the form $(r_4,0)$ with $-r_2<r_4<L_1(r_2)$ satisfying  $m_3=h_3(r_4)$
\item
One of the form $(r_4,0)$ with $L_2(r_2)<r_4<L_2(r_3)$ satisfying  $m_3=h_4(r_4)$
\item
Two of the form $(r_4,\pm r_5)$ with $r_5\ne0$ satisfying  $q_4=q_5$
\end{itemize}

\end{thm}

 \begin{figure}[h]
\vskip.2cm\includegraphics[scale=0.59]{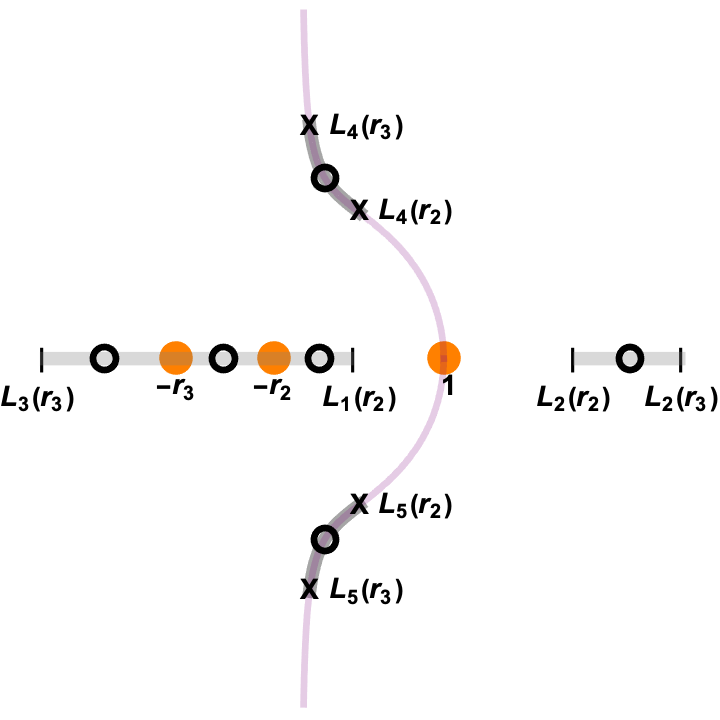}
\caption{There are 4 collinear and two non collinear Lagrange points for each Euler solution. The image shows the possible location of these 6 Lagrange points in terms of the Lagrange .}
\label{LagrangeE}
\end{figure}

\end{document}